\documentclass[12pt,reqno]{amsart}
\usepackage{latexsym}
\usepackage{amssymb}
\usepackage{amscd}
\numberwithin{equation}{section}

\newtheorem{theorem}{Theorem}
\newtheorem{proposition}[theorem]{Proposition}
\newtheorem{lemma}[theorem]{Lemma}
\newtheorem{corollary}[theorem]{Corollary}

\theoremstyle{remark}

\newtheorem*{remark}{Remark}
\newtheorem*{remarks}{Remarks}

\def\om{\omega}

\newcommand{\dbin}[2]{\left(\kern-0.4em{\binom#1#2}
\kern-0.4em\right)}
\begin{document}
\title{A note on higher-dimensional magic matrices}
\newbox\Aut
\setbox\Aut\vbox{
\centerline{\sc Peter J. Cameron$^\dagger$,
Christian Krattenthaler$^{\ddagger}$, \rm and \sc Thomas W. M\"uller$^\dagger$}
\vskip18pt
\centerline{$^\dagger$ School of Mathematical Sciences,}
\centerline{Queen Mary \& Westfield College, University of London,}
\centerline{Mile End Road, London E1 4NS, United Kingdom.}
\centerline{\footnotesize WWW: \tt http://www.maths.qmw.ac.uk/\lower0.5ex\hbox{\~{}}pjc/}
\centerline{\footnotesize WWW: \tt http://www.maths.qmw.ac.uk/\lower0.5ex\hbox{\~{}}twm/}
\vskip18pt
\centerline{$^\ddagger$ Fakult\"at f\"ur Mathematik, Universit\"at Wien,}
\centerline{Nordbergstra\ss e 15, A-1090 Vienna, Austria.}
\centerline{\footnotesize WWW: \footnotesize\tt
http://www.mat.univie.ac.at/\lower0.5ex\hbox{\~{}}kratt} 
}
\author{\box\Aut}

\address{School of Mathematical Sciences, Queen Mary
\& Westfield College, University of London,
Mile End Road, London E1 4NS, United Kingdom.\newline
WWW: {\tt http://www.maths.qmw.ac.uk/\lower0.5ex\hbox{\~{}}pjc/},
\tt http://www.maths.qmw.ac.uk/\lower0.5ex\hbox{\~{}}twm/.}

\address{Fakult\"at f\"ur Mathematik, Universit\"at Wien,
Nordbergstra{\ss}e~15, A-1090 Vienna, Austria.
WWW: \tt http://www.mat.univie.ac.at/\lower0.5ex\hbox{\~{}}kratt.}



\thanks{$^\ddagger$Research partially supported by the Austrian
Science Foundation FWF, grants Z130-N13 and S9607-N13,
the latter in the framework of the National Research Network
``Analytic Combinatorics and Probabilistic Number Theory"}

\begin{abstract}
We provide exact and asymptotic formulae for the number of
unrestricted, respectively indecomposable, $d$-dimensional matrices
where the sum of all matrix entries with one coordinate fixed equals
$2$.
\end{abstract}

\maketitle

\section{Introduction}

We begin by recalling the notion of a {\it magic 
matrix}:\footnote{Strictly speaking, the correct term here would be
``$s$-semi-magic," since we do not require diagonals to sum up to the
same number as the rows and columns, see e.g.\ \cite{BCCGAA}. 
However, here and in what follows we prefer the term ``magic" 
for the sake of brevity.} 
this is a square matrix $m=(m_{i,j})_{1\le i,j\le n}$ with
non-negative integral entries such that all row and column sums are 
equal to
the same non-negative integer. If this non-negative integer is $s$,
then we call such a matrix {\it$s$-magic}.
The enumeration of $s$-magic squares has a long history, going back
at least to
MacMahon \cite[\S404--419]{MacMAA}. A good account of the enumerative theory
of magic squares can be found in 
\cite[Sec.~4.6]{StanAP}, with many pointers to further literature.
For more recent work, see for instance \cite{BCCGAA,LoLYAA}.

Let $[n]$ denote the standard $n$-set $\{1,2,\dots,n\}$.
There are two obvious ways of generalising $s$-magic matrices to higher
dimensions:
\begin{enumerate}
\item[(G1)] {\it All line sums are equal.} Given a positive integer $d$,
a $d$-dimensional matrix $m:[n]^d\to \mathbb N_0$ (where $\mathbb N_0$
denotes the set of non-negative integers) is called {\it $s$-magic}
if
\begin{equation} \label{eq:magic1}
\sum_{\omega_i\in[n]}
m(\omega_1,\omega_2,\dots,\omega_d)=s
\end{equation}
for all fixed $\omega_1,\dots,\omega_{i-1},
\omega_{i+1},\dots,\omega_{d}\in[n]$, and all $i=1,2,\dots,d$.
\item[(G2)] {\it All hyperplane sums are equal.} Given a positive integer $d$,
a $d$-dimensional matrix $m:[n]^d\to \mathbb N_0$ is called {\it $s$-magic}
if
\begin{equation} \label{eq:magic2}
\sum_{\omega_1,\dots,\omega_{i-1},
\omega_{i+1},\dots,\omega_{d}\in[n]}
m(\omega_1,\omega_2,\dots,\omega_d)=s
\end{equation}
for all fixed $\omega_i\in[n]$, and all $i=1,2,\dots,d$.
\end{enumerate}

Generalisation (G1) appears already in the literature, 
see e.g.\ \cite{AhLHAA,BCCGAA}.
For $d=3$ and $s=1$, these objects are equivalent to Latin squares counted
up to isotopy: the roles of rows, columns, and symbols of the
corresponding Latin square are played by the first, second, and third 
coordinate, respectively, and the entry in position
$(\omega_1,\omega_2)$ of the Latin square is $\omega_3$ if and
only if $m(\omega_1,\omega_2,\omega_3)=1$.


Generalisation~(G2) appears in the literature (in more general form) as
\emph{contingency tables} in statistics; there are Markov chain methods
for approximate counting of these, as well as some remarkable asymptotic
estimates, see~\cite{n1,n3,n4,n5,n6}. Indeed, these results suggest
that the counting problem for (G2) is much easier than for (G1).
(We are grateful to a referee for this information and the references.)

The present note focusses on the second generalisation. Hence, from
now on, whenever we use the term ``$s$-magic," this is understood in the
sense of (G2). 

Counting higher-dimensional magic matrices is made more difficult
(than the already difficult case of $2$-dimensional magic matrices)
by the fact that the analogue of Birkhoff's Theorem 
(cf.\ \cite{Birkhoff} or 
\cite[Corollary~8.40]{Aig}; it says that any $2$-dimensional
$s$-magic matrix can be decomposed in a sum of permutation matrices,
that is, $1$-magic matrices) fails for
them. For example, the $3$-dimensional $2$-magic matrix with ones
in positions $(1,1,1)$, $(1,2,3)$, $(2,1,2)$, $(2,2,1)$, $(3,3,2)$
and $(3,3,3)$ is not the sum of two $1$-magic matrices. 

As we demonstrate in this note, it is however
possible to count the $2$-magic matrices of any dimension.
Our first result is a recurrence relation for the number $u_n(d)$
of indecomposable $d$-dimensional $2$-magic matrices of size~$n$
(see Corollary~\ref{coro} in Section~\ref{sec:rek}). 
This recurrence is used in Proposition~\ref{Prop:HighDimAsymp} 
to derive, for fixed $d\ge3$, an asymptotic formula for
$u_n(d)$. In order to go from indecomposable matrices to unrestricted ones,
we observe that the 
$d$-dimensional $2$-magic matrices may be viewed as a $d$-sort
species in the sense of Joyal \cite{Joyal} which obeys the
($d$-sort) exponential principle. 
Let $w_n(d)$ denote the number of {\it all\/} $d$-dimensional 
$2$-magic matrices of size~$n$. The exponential principle can then be
applied  to relate the numbers $w_n(d)$ to the numbers $u_n(d)$, see
\eqref{eq:wugf} (for $d=2$) and \eqref{eq:wnun} (for $d\ge2$).
This relation is used in Theorem~\ref{Th:HighDimAsymp} to find, for fixed
$d\ge3$, an asymptotic estimate for the numbers $w_n(d)$ as well. 
Exact and asymptotic formulae for
$u_n(d)$ and $w_n(d)$ for $d=2$ are presented in
Section~\ref{sec:d=2}. 
We remark in passing that a simple counting argument shows that
the obvious interpretation of the matrices
in Generalisation~(G1) as a $d$-sort species does {\it not\/} satisfy the
exponential principle, not even under the --- in a sense --- minimal
axiomatics of \cite{CaKMAA}.

\section{Indecomposable $2$-magic matrices and fixed-point-free involutions}
\label{sec:vn}

A $d$-dimensional matrix $m:[n]^d\to\mathbb N_0$ is called {\it
decomposable}, if there exist non-empty subsets
$B_1^{(1)},B_2^{(1)},B_1^{(2)},B_2^{(2)},\dots,B_1^{(d)},B_2^{(d)}$
of $[n]$ with
$$
B_1^{(1)}\amalg B_2^{(1)}=B_1^{(2)}\amalg B_2^{(2)}=\dots=
B_1^{(d)}\amalg B_2^{(d)}=[n]
$$
($\amalg$ denoting disjoint union) and
$$
\vert B_1^{(1)}\vert=\vert B_1^{(2)}\vert=\dots=\vert B_1^{(d)}\vert,
$$
such that $m(\om_1,\om_2,\dots,\om_d)\ne0$ only if either
$$(\om_1,\om_2,\dots,\om_d)
\in B_1^{(1)}\times B_1^{(2)}\times\dots\times B_1^{(d)}$$
or
$$(\om_1,\om_2,\dots,\om_d)
\in B_2^{(1)}\times B_2^{(2)}\times\dots\times B_2^{(d)},$$
otherwise it is called {\it indecomposable}.\footnote{We warn 
the reader that for $d=2$ this does not reduce to
the notion of decomposability of matrices in linear algebra
since there rows and columns are reordered by the {\it same}
permutation. Yet another definition of indecomposability occurs
in~\cite{AhLHAA}.}
(In less formal language: there exist reorderings of the lines of
the matrix such that $m$ attains a block form.) The integer $n$ is
called the {\it size} of $m$. 

Let $u_n(d)$ denote the number of indecomposable $d$-dimensional
$2$-magic matrices of size $n$. Note that
an indecomposable $2$-magic matrix with an entry~$2$ has size~$1$.
So it is enough to consider zero-one matrices.

The purpose of this section is to relate the numbers $u_n(d)$ to
another sequence of numbers $v_n(d)$ counting certain tuples of
fixed-point-free involutions on a set with $2n$ elements.
More precisely, let 
\begin{equation} \label{eq:t_1}
t_1=(1,2)(3,4)\cdots(2n-1,2n)
\end{equation}
be the standard fixed-point-free involution on the set $[2n]$.
Then we define $v_n(d)$ to be the number of choices of $d-1$
fixed-point-free involutions $t_2, \ldots, t_d$ on
$[2n]$ such that the group $G=\langle t_1, t_2, \ldots,
t_d\rangle$ generated by $t_1,t_2,\dots,t_d$ is transitive. 
(For example, when $n=2$, there are just three fixed-point-free
involutions on $\{1,2,3,4\}$, viz., $(1,2)(3,4)$, $(1,3)(2,4)$ and
$(1,4)(2,3)$, any two of which generate a transitive group. So
$v_2(d)=3^{d-1}-1$.)

We have the following relation.

\begin{lemma} \label{lem:unvn}
For all integers $n,d>1$, we have
\begin{equation} \label{eq:unvn}
u_n(d) =2^{-n}(n!)^{d-1}v_n(d).
\end{equation}
\end{lemma}

\begin{proof} 
Let $m$ be an indecomposable $d$-dimensional $2$-magic
matrix of size~$n$, where $n>1$. Then $m$ is a zero-one matrix,
and
it contains $2n$ entries equal to~$1$, the rest being zero. Number
the positions of the $1$'s in $m$ from $1$ to $2n$ in such a way
that the positions with first coordinate $j$ are numbers $2j-1$
and $2j$ for $j=1,\ldots,n$. (There are $2^n$ ways to do this, since
for each $j$ we can choose arbitrarily which of the two $1$'s has number
$j-1$.) Then, for $i=1,\ldots,d$, let $t_i$ be the fixed-point-free
involution whose cycles are the pairs of numbers in
$\{1,\ldots,2n\}$ indexing positions of $1$'s with the same $i$-th
coordinate. Note that $t_1$ is the involution defined in \eqref{eq:t_1}.

We claim that the subgroup $G$ of $S_{2n}$ generated by $t_1,\ldots,t_d$
is transitive if and only if the matrix $m$ is indecomposable. For this,
note that the $1$'s whose labels belong to a cycle of $t_i$ have the same
$i$-th coordinate. So, if $m$ is decomposable, and the $1$ with label
$1$ belongs to $B_1^{(1)}\times\cdots\times B_1^{(d)}$, then
an easy induction shows that any $1$ whose label is in the same orbit
belongs to this set, so that $G$ is intransitive. Conversely, if $G$ is
intransitive, then the coordinates of the $1$'s whose labels belong to a
$G$-orbit give rise to a decomposition of $m$. 

So each matrix gives rise to $2^n$ such $d$-tuples of involutions. Thus,
the number of pairs consisting of a matrix and a corresponding
sequence of permutations is $2^n\,u_n(d)$.

For instance, the example of a matrix failing the analogue of
Birkhoff's Theorem given in the Introduction, 
with the entries numbered in the
order given, produces the three permutations $(1,2)(3,4)(5,6)$,
$(1,3)(2,4)(5,6)$ and $(1,4)(2,6)(3,5)$.

Conversely, let $t_1, \ldots, t_d$ be fixed-point-free involutions
on the set $\{1, \ldots, 2n\}$ which generate a transitive group,
where $t_1$ is the standard involution defined in \eqref{eq:t_1}. Number the
cycles of each $t_i$ from $1$ to $n$ such that the cycle
$(2j-1,2j)$ of $t_1$ has number~$j$. (There are $(n!)^{d-1}$ such
numberings.) Now construct a $d$-dimensional matrix $m$ as
follows: for $k=1,\ldots,2n$, if $k$ lies in cycle number~$\om_i$
of~$t_i$, then $m(\om_1,\om_2,\ldots,\om_d)=1$; all other entries are
zero. Then $m$ is $2$-magic. Consequently, each sequence of
permutations gives rise to $(n!)^{d-1}$ matrices; and the number
of pairs consisting of a matrix and a corresponding sequence of
permutations equals $(n!)^{d-1}\,v_n(d)$.

Comparing these two expressions, we obtain
\eqref{eq:unvn}, as required.
\end{proof}

\begin{remark}
We note that $u_1(d)=v_1(d)=1$ for all $d$. Hence,
Formula~\eqref{eq:unvn} is false for $n=1$.
\end{remark}

\section{Computation of $u_n(2)$ and $w_n(2)$}
\label{sec:d=2}

The number $w_n(2)$ of $2$-dimensional $2$-magic matrices of size $n$ has 
been addressed earlier by Anand, Dumir and Gupta in \cite[Sec.~8.1]{ADG}.
They found the generating function formula
\begin{equation} \label{eq:wgf}
\sum _{n\ge0} ^{}w_n(2)\frac {z^n} {(n!)^2}=
(1-z)^{-1/2}e^{z/2}.
\end{equation}
This gives the explicit formula
\begin{equation} \label{eq:wexpl}
w_n(2)=\sum _{k=0} ^{n}\binom {2k}k \frac {(n!)^2} {2^{n-k}(n-k)!}.
\end{equation}
Singularity analysis (cf.\ \cite[Ch.~VI]{FlSeAA}) applied to
\eqref{eq:wgf} then yields the asymptotic formula
\begin{equation} \label{eq:wasy}
w_n(2)=(n!)^2\sqrt{\frac {e} {\pi n}}
\left(1+\mathcal O\left(\frac {1} {n}\right)\right),
\quad \quad \text {as }n\to\infty.
\end{equation}

The number $u_n(2)$ of {\it indecomposable} $2$-dimensional $2$-magic 
matrices of size $n$ can also be computed explicitly. One way is
to observe that, by Birkhoff's Theorem (cf.\ \cite{Birkhoff} or 
\cite[Corollary~8.40]{Aig}), 
a $2$-magic matrix $m$ is the sum of two permutation
matrices, say $p_1$ and $p_2$. 
If $m$ is indecomposable, then the
pair $\{p_1,p_2\}$ is uniquely determined. Premultiplying by
$p_1^{-1}$, we obtain a situation where $p_1$ is the identity;
indecomposability forces $p_2$ to be the permutation matrix
corresponding to a cyclic permutation, since a cycle of $p_2$
not containing all points would provide a decomposition of $m$.
So there are $n!\,(n-1)!$
choices for $(p_1,p_2)$, and half this many choices for $m$
(assuming, as we may, that $n>1$). Note that this formula gives
half the correct number for $n=1$. So we have
\begin{equation} \label{eq:un}
u_n(2) = \begin{cases}
1, & \text{if }n=1,\\ \frac {1} {2}n!\,(n-1)!, & \text{if }n>1.
\end{cases}
\end{equation}
Alternatively, we may observe that $2$-dimensional $2$-magic matrices
may be seen as a $2$-sort species in the sense of Joyal 
\cite{Joyal} (see also \cite[Def.~4 on p.~102]{BLL}), with the row
indices and the column indices forming the two sets on which the
functor defining the species operates. Hence, by the exponential
principle for $2$-sort species \cite[Prop.~20]{Joyal} (see also 
\cite[Sec.~2.4]{BLL}), 
we have
\begin{equation} \label{eq:wugf}
\sum _{n\ge0} ^{}w_n(2)\frac {z^n} {(n!)^2}=
\exp\left(\sum _{n\ge1} ^{}u_n(2)\frac {z^n} {(n!)^2}
\right).
\end{equation}
Combining this with \eqref{eq:wgf}, we find that
$$
\sum _{n\ge1} ^{}u_n(2)\frac {z^n} {(n!)^2}=
\frac {z} {2}+\frac {1} {2}\log\left(\frac {1} {1-z}\right).
$$
Extraction of the coefficient of $z^n$ then leads (again) to \eqref{eq:un}.

\section{A recurrence relation for $v_n(d)$}
\label{sec:rek}

In this section we prove a recurrence relation for the numbers $v_n(d)$
(see Section~\ref{sec:vn} for their definition). By
Lemma~\ref{lem:unvn}, this affords as well a recurrence relation for
the numbers $u_n(d)$.

\begin{proposition} \label{prop:vrek}
The numbers $v_n(d)$ satisfy $v_1(d)=1$ and
\begin{equation} \label{eq:vrek}
\sum_{k=1}^n{\binom {n-1}{ k-1}}\,((2n-2k-1)!!)^{d-1}\,v_k(d)  =
((2n-1)!!)^{d-1},\quad n>1.
\end{equation}
Here, $(2n-1)!!=1\cdot3\cdot5\cdots(2n-1)$ is the product of the
first $n$ odd positive integers for $n>0$, and, by convention,
$(-1)!!=1$.
\end{proposition}

\begin{proof} Recall that $(2n-1)!!$ is the number of
fixed-point-free involutions on a set of size~$2n$. 
(This is a special case of the general formula
\[\frac{n!}{\prod_{i=1}^ni^{a_i}a_i!}\]
for the number of permutations in $S_n$ with $a_i$ cycles of length $i$ for
$i=1,\ldots,n$.) The number of choices of
involutions $t_1,t_2,\dots,t_d$, where $t_1$ is as in \eqref{eq:t_1},
such that the orbit containing~$1$ of the group they generate
has size $2k$ is\[
{\binom {n-1}{ k-1}}\,((2n-2k-1)!!)^{d-1}\,v_k(d),
\]
since we can choose in order
\begin{enumerate}
\item[(i)] $k-1$ of the $n-1$ cycles of $t_1$ other than $(1,2)$ 
such that the elements not fixed by all of these $k-1$ transpositions
together with $\{1,2\}$ form the desired orbit, $O$ say;
\item[(ii)] $d-1$ fixed-point-free involutions on $O$ which, together
with the restriction of $t_1$
to $O$, generate a transitive group;
\item[(iii)] $d-1$ arbitrary fixed-point-free involutions on the complement of
$O$.
\end{enumerate}
Summing these values shows that the numbers $v_n(d)$ satisfy the
desired recurrence.
\end{proof}

\begin{corollary} \label{coro}
For all integers $d>1$, the numbers $u_n(d)$ satisfy $u_1(d)=1$ and
\begin{multline*}
((2n-3)!!)^{d-1}+
\sum_{k=2}^n{\binom {n-1}{ k-1}}\,\left(\frac 
{(2n-2k-1)!!} {k!}\right)^{d-1}2^k\,u_k(d)  =
((2n-1)!!)^{d-1},\\\quad n>1.
\end{multline*}
\end{corollary}

\begin{remarks}
(1)
In the case $d=2$, we have seen in \eqref{eq:un} that $u_n(2)=n! (n-1)!/2$
for $n>1$, so that
\[
v_n(2)=2^{n-1}\,(n-1)! = (2n-2)!!,
\]
where $(2n-2)!!$ is the product of the even integers up to $2n-2$
(with $0!!=1$ by convention). Substituting this in \eqref{eq:vrek},
we have proved the somewhat curious looking identity
\[
\sum_{k=1}^n{\binom {n-1}{ k-1}}\,(2n-2k-1)!!\,(2k-2)!!  =
(2n-1)!!
\]
for $n>1$.

We remark that this identity has an interpretation in terms of
hypergeometric functions, for which we refer to \cite{SlatAC}, in
particular, (1.7.7), Appendix~(III.4). The left-hand side is
$$2^{n-1}\,(1/2)_{n-1}
\cdot  {} _{2} F _{1} \!\left [ \begin{matrix} {  - n +1,1}\\ 
{ -n+ {\frac 1 2}}\end{matrix} ; {\displaystyle
      1}\right ],
$$
and the identity is an instance of the Chu--Vandermonde identity.

\medskip
(2)
For $d>2$, we have not been able to solve the recurrence
explicitly. However, it is easy to calculate terms in the
sequences, and we can describe their asymptotics (see
Sections~\ref{sec:uasy} and \ref{sec:wasy}).

Table~\ref{tab} gives counts of all indecomposable matrices, all
zero-one matrices, and all non-negative integer matrices, with
dimension~$d$ and hyperplane sums~$2$. The sequences for $d=2$ are
numbers A010796, A001499, and A000681 in the On-Line Encyclopedia
of Integer Sequences~\cite{oeis}. For $d=3$, they are A112578,
A112579 and A112580.

\begin{table}
\begin{tabular}{|c|c|rrrrrr|}
\hline $d$ && $n=1$ & $n=2$ & $n=3$ & $n=4$ & $n=5$ & $n=6$ \\
\hline
2 & \hbox{indec} &1 & 1 & 6 & 72 & 1440 & 43200 \\
  & \hbox{$0$-$1$} & 0 & 1 & 6 & 90 & 2040 & 67950 \\
  & \hbox{all} & 1 & 3 & 21 & 282 & 6210 & 202410 \\
\hline
3 & \hbox{indec} & 1 & 8 & 900 & 359424 & 370828800 & 820150272000 \\
  & \hbox{$0$-$1$} & 0 & 8 & 900 & 366336 & 378028800 & 833156928000  \\
  & \hbox{all} & 1 & 12 & 1152 & 431424 & 427723200 & 920031955200 \\
\hline
\end{tabular}
\vskip10pt
\caption{Indecomposable, zero-one and arbitrary $d$-dimensional $2$-magic
matrices of size~$n$}
\label{tab}
\end{table}
\end{remarks}

\section{Asymptotics of the numbers $u_n(d)$}
\label{sec:uasy}

This section provides the preparation for the determination of the
asymptotics of the numbers $w_n(d)$ for $d\ge3$ in the next section.
Our goal here is to establish an asymptotic estimate for the sequence
$u_n(d)$ with fixed $d\ge3$.

\begin{proposition}
\label{Prop:HighDimAsymp} For fixed $d\ge3$, 
we have
\[
u_n(d)\sim 2^{-dn}((2n)!)^{d-1},\quad \quad 
\text {as }n\to\infty.
\]
\end{proposition}

\begin{proof}
By Lemma~\ref{lem:unvn}, we have
$u_n(d)=(n!)^{d-1}v_n(d)/2^n$ for $n>1$, so it suffices to show that
\[
v_n(d)\sim((2n-1)!!)^{d-1}.
\]
We will use the estimates
\[
\sqrt{2(n+1)} \le \frac{2^n\,n!}{(2n-1)!!} \le 2\sqrt{n}
\]
for $n\ge1$. With $c_n=2^nn!/(2n-1)!!$, we have
$c_{n+1}/c_n=(2n+2)/(2n+1)$, and both inequalities are easily
proved by induction. From these estimates, we obtain the
inequality
\begin{equation} \label{eq:estimate}
\frac{(2n-1)!!}{(2k-1)!!\,(2n-2k-1)!!} \ge {\binom n
k}\left(\frac{(k+1)(n-k+1)}{n}\right)^{1/2}.
\end{equation}
To simplify our formulae, we denote the left-hand side of this
inequality by $\displaystyle{\dbin{n}{k}}$.

Now, by Proposition~\ref{eq:vrek}, $v_n(d)$ satisfies the recurrence
$$
v_n(d)=((2n-1)!!)^{d-1} - \sum_{k=1}^{n-1}{\binom {n-1}
{k-1}}((2n-2k-1)!!)^{d-1}v_k(d), \quad \quad n>1.
$$
Clearly $v_n(d)\le((2n-1)!!)^{d-1}$. We show that
$v_n(d)\ge((2n-1)!!)^{d-1}(1-O(1/n))$,
an estimate which, in view of the above recurrence, follows if
we can show that
\[
L:=\sum_{k=1}^{n-1}{\binom {n-1}
{k-1}}\dbin{n}{k}^{-(d-1)}=O\left(\frac{1}{n}\right).
\]

Using
\eqref{eq:estimate}, we have
\begin{align*}
L &\le 
\frac{n}{(2n-1)^{d-1}} 
+ \sum_{k=2}^{n-2}{\binom {n-1}
{k-1}}{\binom n k}^{-(d-1)}
\left(\frac{n}{(k+1)(n-k+1)}\right)^{(d-1)/2} \\ 
&\le
\frac{n}{(2n-1)^{d-1}} 
+ \sum_{k=2}^{n-2}\frac{k}{n}{\binom n
k}^{-(d-2)} \left(\frac{n}{(k+1)(n-k+1)}\right)^{(d-1)/2}.
\end{align*}
Since $k/n<1$,
$n/(k+1)(n-k+1)<1/2$, and ${\binom n
k}\ge{\binom n 2}$, and there are fewer than $n-1$ terms in the
sum, the second term is at most
$$n^{-(d-2)}(n-1)^{-(d-3)}\cdot 2^{d-2}\cdot 2^{-(d-1)/2}\le 
\frac {1} {n},$$
as required.
\end{proof}

\section{Asymptotics of the numbers $w_n(d)$}
\label{sec:wasy}

Recall that $w_n(d)$ and $u_n(d)$ are the numbers of unrestricted,
respectively indecomposable, $d$-dimensional $2$-magic matrices of
size $n$. Using the exponential principle, we can relate 
the sequence $(w_n(d))_{n\ge0}$ to the sequence $(u_n(d))_{n\ge0}$
for each fixed $d$, see \eqref{eq:wnun} below. This relationship combined with the
fact that the sequence $(u_n(d))_{n\ge0}$ grows sufficiently rapidly
for $d\ge3$ (Proposition~\ref{Prop:HighDimAsymp} says that it grows
very roughly like $((2n)!)^{d-1}$) allows us to conclude that, for
$d\ge3$, $w_n(d)$ and $u_n(d)$ grow at the same rate.

\begin{theorem}
\label{Th:HighDimAsymp} For fixed $d\ge3$, we have
\[
w_n(d)\sim 2^{-nd}((2n)!)^{d-1},\quad \quad 
\text {as }n\to\infty.
\]
\end{theorem}

\begin{proof}
Generalising the argument at the end of Section~\ref{sec:d=2},
we observe that $d$-dimensional $2$-magic matrices
may be seen as a $d$-sort species in the sense of Joyal 
\cite{Joyal} (see also \cite[Def.~4 on p.~102]{BLL}), with the row
indices and the column indices forming the two set on which the
functor defining the species operates. Hence, by the exponential
principle for $d$-sort species \cite[Prop.~20]{Joyal} (see also 
\cite[Sec.~2.4]{BLL}), 
we have
$$
\sum _{n\ge0} ^{}w_n(d)\frac {z^n} {(n!)^d}=
\exp\left(\sum _{n\ge1} ^{}u_n(d)\frac {z^n} {(n!)^d}
\right).
$$
If we now differentiate both sides of this equation with respect to
$z$ and subsequently multiply both sides by $z$, then we obtain
\begin{align*}
\sum _{n\ge0} ^{}nw_n(d)\frac {z^n} {(n!)^d}
&=
\left(\sum _{n\ge1} ^{}nu_n(d)\frac {z^n} {(n!)^d}\right)
\exp\left(\sum _{n\ge1} ^{}u_n(d)\frac {z^n} {(n!)^d}\right)\\
&= 
\left(\sum _{n\ge1} ^{}nu_n(d)\frac {z^n} {(n!)^d}\right)
\left(\sum _{n\ge0} ^{}w_n(d)\frac {z^n} {(n!)^d}\right).
\end{align*}
Comparison of coefficients of $z^n$ on both sides then leads to the
relation
\begin{equation} \label{eq:wnun}
w_n(d)=u_n(d)+\sum_{k=1}^{n-1}\frac{k}{n}{\binom n
k}^du_k(d)w_{n-k}(d).
\end{equation}

As we said at the beginning of this section, our goal is to show that
$w_n(d)$ grows asymptotically at the same rate as $u_n(d)$. Hence,
putting $w_n(d)=u_n(d)+x_n(d)$, we have to show that $x_n(d)=o(u_n(d))$. We
assume inductively that 
$$x_m(d)\le 2^{-m}((2m-1)!!)^{d-1}(m!)^{d-1}$$
for all $m$ between $2$ and $n-1$;
the induction starts since we have $x_1(d)=x_2(d)=0$.

Now, using the inductive hypothesis with the recurrence relation
\eqref{eq:wnun},
we have
\begin{align*}
\frac{x_n(d)2^n}{((2n-1)!!)^{d-1}(n!)^{d-1}} &\le 2 \sum_{k=1}^{n-1}
\frac{k}{n} {\binom n k}^d \dbin{n}{k}^{-(d-1)}{\binom n
k}^{-(d-1)}
\\
&\le 2 \sum_{k=1}^{n-1} \frac{k}{n}{\binom n k}^{-(d-2)}
\left(\frac{n}{(k+1)(n-k+1)}\right)^{(d-1)/2}
\\
&\le (2^{1/2}n)^{-(d-3)},
\end{align*}
which establishes the result if $d>3$. For $d=3$, this inequality
gives the inductive step (that is, that the left-hand side is at
most~$1$); the fact that it is $o(1)$ for large~$n$ is proved by
an argument like that in the proof of
Proposition~\ref{Prop:HighDimAsymp}.
\end{proof}

\paragraph{Acknowledgment} The authors are grateful to a referee for
some helpful information on contingency tables.

\end{document}